\documentclass[12pt, a4paper]{amsart}
\usepackage{epsfig}
\usepackage{amssymb}
\usepackage{amsmath}
\usepackage{latexsym,graphicx, color}
\newtheorem{thm}{Theorem}[section] \newtheorem{lem}[thm]{Lemma}
\newtheorem{cor}[thm]{Corollary} 
\newtheorem{prop}[thm]{Proposition} \theoremstyle{definition}

\usepackage{hyperref}

\def\co{\colon\thinspace}
\usepackage{color}  

\begin{document}

\date{\today}
\author{Michael Heusener and Richard Weidmann}
\title{Generating pairs of 2-bridge knot groups}

\begin{abstract}
We study Nielsen equivalence classes of generating pairs of Kleinian
groups and HNN-extensions. We establish the following facts:

\begin{enumerate}
\item Hyperbolic 2-bridge knot groups have infinitely many Nielsen classes of generating pairs.
\item For any $n\in\mathbb N$ there is a closed hyperbolic 3-manifold whose fundamental group has $n$ distinct Nielsen classes of generating pairs.
\item Two pairs of elements of a fundamental group of an HNN-extension are Nielsen equivalent iff they are so for the obvious reasons.
\end{enumerate}
\end{abstract}

\maketitle

\section*{Introduction}

The main purpose of this note is to study Nielsen equivalence classes of generating pairs of fundamental groups of hyperbolic 2-bridge knot spaces and of closed hyperbolic 3-manifolds obtained from those spaces by Dehn fillings.

\smallskip It is a result of Delzant (following Gromov) \cite{D} that any
torsion-free hyperbolic group has only finitely many Nielsen classes of
generating pairs. In the case of closed hyperbolic 3-manifolds Delzant's proof
actually provides an explicit upper bound for this number in terms of the
injectivity radius as observed by Agol, see  \cite{Sou1} for an account of
Agol's ideas. The finiteness of
Nielsen classes of
generating tuples of fundamental groups of closed hyperbolic 3-manifolds of arbitrary size was established in \cite{KW}.

\smallskip The examples constructed in this article show that this finiteness fails for cusped hyperbolic 3-manifolds. We establish the following:

\begin{thm}\label{thma}
Let $\mathfrak k$ be a hyperbolic $2$-bridge knot with knot exterior $M$. Then $\pi_1(M)$ has infinitely many Nielsen classes of generating pairs. 
\end{thm}

We further show that there is no uniform bound on the number of Nielsen classes of generating pairs of fundamental groups of closed hyperbolic 3-manifolds if the assumption on the injectivity radius is dropped. The constructed manifolds are obtained from hyperbolic 2-bridge knot complements by increasingly complicated Dehn fillings.

\begin{thm}\label{thmb} For any $n$ there exists a closed hyperbolic 3-manifold $M$ such that $\pi_1(M)$ has at least $n$ distinct Nielsen classes of generating pairs.
\end{thm}

The non-uniqueness of Nielsen classes of generating tuples of fundamental
groups of hyperbolic 3-manifolds is not new. It is an immediate consequence of
the work of Lustig and Moriah \cite{LM} that there exist closed hyperbolic 3-manifolds whose fundamental groups are of rank $r$ and have at least $2^r-2$ Nielsen classes of generating $r$-tuples. Note that while the distinct Nielsen classes exhibited by Lustig and Moriah are geometric and therefore correspond to non-isotopic Heegaard splittings  this is not true in the current setting. Indeed by Kobayashi's work \cite{Ko} it is known that 2-bridge knot exteriors admit at most 6 isotopy classes of Heegaard splittings of genus $2$, thus almost all of the generating pairs exhibited in this note are non-geometric.

\smallskip The proofs rely on some simple facts about Nielsen equivalence of generating pairs due to Nielsen and in the case of Theorem~\ref{thma} some basic hyperbolic geometry. For the proof of Theorem~\ref{thmb} we further exploit the geometric convergence of manifolds obtained by increasingly complicated Dehn surgery on a 2-bridge knot to the hyperbolic knot complement.

\smallskip

After discussing some basic material on Nielsen equivalences of generating pairs we first prove a simple theorem about generating pairs of HNN-extensions. The argument in this case is easier but similar to the argument needed in the proofs of the two theorems discussed above. We will then establish a simple fact about about piecewise geodesics in hyperbolic space before we proceed with the proof of the main theorem.

\smallskip The authors would like to thank Yoav Moriah for his useful comments and Makoto Sakuma for a stimulating discussion. Moreover the authors would like to thank the referee whose numerous suggestions resulted in a greatly improved exposition and a shorter proof of Theorem~\ref{graphsofgroups}.

\section{Nielsen equivalence of pairs of elements}

Let $G$ be a group and $\mathcal T=(g_1,\ldots ,g_k)$ and $\mathcal T'=(h_1,\ldots ,h_k)$ be two $k$-tuples of elements of $G$. We say that $\mathcal T$ and $\mathcal T'$ are {\em elementarily equivalent} if one of the following holds:

\begin{enumerate}
\item $h_i=g_{\sigma(i)}$ for $1\le i\le k$ and some $\sigma\in S_k$.
\item $h_i=g_i^{-1}$ for some $i\in\{1,\ldots ,k\}$ and $h_j=g_j$ for $j\neq i$.
\item $h_i=g_ig_j^{\varepsilon}$ for some $i\neq j$ with $\varepsilon\in\{-1,1\}$ and $h_l=g_l$ for $l\neq i$.
\end{enumerate}

Two tuples are further called {\em Nielsen equivalent} if one can be transformed into the other by a finite sequence of elementary equivalences. Note that the elementary equivalences are also called \emph{Nielsen transformations} or \emph{Nielsen moves}.

\smallskip The fact that $\hbox{Aut } F_n$ is generated by so-called elementary Nielsen automorphisms implies that the above definition of Nielsen equivalence can be rephrased in the following way.

\smallskip Let $F_k=F(x_1,\ldots ,x_k)$ be the free group of rank $k$. Then
two $k$-tuples $\mathcal T=(g_1,\ldots ,g_k)$ and $\mathcal T'=(h_1,\ldots ,h_k)$ are Nielsen equivalent iff there exists a homomorphism $\phi:F_k\to G$ and an automorphism $\alpha$ of $F_k$ such that the following hold:
\begin{enumerate}
\item $g_i=\phi(x_i)$ for $1\le i\le k$.
\item $h_i=\phi\circ\alpha(x_i)$ for $1\le i\le k$.
\end{enumerate}

\medskip Deciding Nielsen equivalence of two tuples or classifying all Nielsen equivalence classes is usually a very difficult problem and undecidable in general. However in the case of pairs of elements the situation tends to be much easier. The reason is that the automorphism group of $F_2$ and the structure of primitive elements in $F_2$ are particularly easy to understand. 

\medskip Nielsen \cite{N} observed that any automorphism of $F(a,b)$ preserves the  commutator $[a,b]=aba^{-1}b^{-1}$ up to conjugation and inversion. This is easily verified by checking that it holds for the elementary Nielsen automorphisms. As a consequence we get the following simple and much used test for Nielsen equivalence of pairs of elements.

\begin{prop}\label{Nielsen} Let $G$ be a group and $(x,y)$ and $(x',y')$ be two pairs of elements. If $(x,y)\sim(x',y')$ then $[x,y]$ is conjugate to $[x',y']$ or $[x',y']^{-1}$.
\end{prop}

While convenient, the above criterion is in general not sufficient to distinguish all Nielsen classes. Another useful fact in distinguishing Nielsen classes of pairs is that primitive elements of $F(a,b)$ are well understood, in fact in \cite{OZ} Osborne and Zieschang gave a complete description of primitive elements of the free groups of rank $2$; recall that an element of a free group or a free Abelian group is called primitive if it is part of some basis. 
The proof in \cite{OZ} relies on the fact already observed by Nielsen \cite{N} that for any primitive element $p$ in the abelianization of $F(a,b)$ there is a unique conjugacy class of primitive elements in $F(a,b)$ that is mapped to $p$.

\smallskip 
An immediate consequence of their description is the proposition below, see also \cite{CMZ}. We give a proof of the weaker statement that we need as we can without breaking a sweat, note that $\varepsilon$ and $\eta$ below are simply the signs of the exponents of $a$ and $b$ in the abelianization of $g$. 

\begin{prop}\label{osbornezieschang} Let $g$ be a primitive element of $F(a,b)$. Then there exist $\varepsilon,\eta\in\{-1,1\}$  such that $g$ is conjugate to an element represented by a positive word in $a^{\varepsilon}$ and $b^{\eta}$.
\end{prop}

\begin{proof} As the proof in \cite{OZ} we rely on the fact that for any primitive element $z_1a+z_2b$  in the abelianization we have a unique conjugacy class of primitive elements in $F(a,b)$ that maps to $z_1a+z_2b$.  

\smallskip We assume that $g$ maps to $na+mb$ in the homology where $n,m\ge 0$, the other cases are analogous. We need to show that there exists a primitive element in $F(a,b)$ that can be written as a positive word in $a$ and $b$ that maps to $na+mb$.

\smallskip Choose $r,s\ge 0$ such that $na+mb$ and $ra+sb$ form a basis of the homology. It is easily checked that we can transform this basis  into the basis $a$, $b$ by only applying elementary Nielsen transformations of type (1) and of type (3) with $\varepsilon=-1$ such that all intermediate elements only have positive coefficients. We recover the original basis by running the inverse transformation  in inverse order, here all transformation are of type (1) or of type (3) with $\varepsilon=1$.

\smallskip Now we can run the same sequence of inverse Nielsen transformations
in $F(a,b)$ starting with $a$, $b$. We obtain a basis of $F(a,b)$ whose first
element maps to $na+nb$ in the homology. As no inverses are introduced in this
sequence of Nielsen transformations it follows that this first basis element
is a positive word in $a$ and $b$ and must be conjugate to $g$. This proves the claim.
\end{proof}



\section{Generating pairs of HNN-extensions}

In the following we assume that $\mathbb A$ is a graph of groups with underlying graph $A$. The vertex group of a vertex $v$ is denoted by $A_v$. It is well known that any tuple that  generates a non-free subgroup is Nielsen equivalent to a tuple containing an elliptic element, i.e. an element conjugate to an element of one of the vertex groups, see \cite{St}, \cite{Z}, \cite{PR} and \cite{W1} for various levels of generality. The tuple containing an elliptic element can be obtained from the original one by a sequence of length reducing Nielsen moves. If the underlying graph is not a tree, i.e. if ${\mathbb A}$ has an HNN-component then it is not possible that all generators are elliptic as they would all lie in the kernel of the projection to the fundamental group of the underlying graph.

\smallskip This justifies in the theorem below to only consider pairs of elements $(x,y)$ such that $x$ is elliptic and $y$ is hyperbolic. 

\begin{thm}\label{graphsofgroups} Let $\mathbb A$ be a graph of groups whose underlying graph $A$ is not a tree. Let $G=\pi_1(\mathbb A)$. Suppose that $(x,y)$ and $(x',y')$ are generating pairs of $G$  such that $x$ and $x'$ are elliptic. 

Then $(x,y)\sim(x',y')$ iff there exist $g\in G$, $k\in\mathbb Z$ and $\varepsilon,\eta\in\{-1,1\}$ such that $$x'=gx^\varepsilon g^{-1}\hbox{ and }y'=gy^\eta x^kg^{-1}.$$\end{thm}

\begin{proof} 
For any graph of groups $\mathbb A$ with underlying graph $A$ there is a natural epimorphism $\pi:\pi_1(\mathbb A)\to\pi_1(A)$ whose kernel is generated by the elliptic elements. In our context $\pi_1(A)$ is generated by $\pi(y)$ since $x$ lies in the kernel of $\pi$. Thus $\pi_1(A)\cong \mathbb Z$ as $A$ is not a tree, in particular $\pi(y)$ is an element of infinite order.


Suppose now that $w$ is a positive word in $x^{\varepsilon}$ and $y^{\eta}$ for fixed
$\varepsilon,\eta\in\{-1,1\}$ with at least one occurence of $y^{\eta}$. Let $N$ be the number of occurences of $y^{\eta}$ in $w$. By assumption $N>0$ and as $w$ is a positive word we have $\pi(w)=\pi(y^\eta)^N$, in particular $w\notin\ker\pi$ and $w$ cannot be elliptic.

Let now $\psi:F(a,b)\to G$ be the epimorphism given by $a\mapsto x$ and $b\mapsto y$. If $(x,y)$ is Nielsen equivalent to $(x',y')$ then $x'=\psi(h)$ for some primitive element $h\in F(a,b)$. By Proposition~\ref{osbornezieschang} $h$ is conjugate to some element represented by a positive word $w$ in $a^\varepsilon$ and $b^\eta$ for some $\varepsilon,\eta\in\{-1,1\}$. As $\psi(w)$ is elliptic it follows from the above that $w$ is a positive word in $a^\varepsilon$. As $w$ represents a primitive element of $F(a,b)$ it follows that $w=a^\varepsilon$. This implies that $x'=gx^{\varepsilon}g^{-1}$ for some $g\in G$.

The second part is immediate. Indeed any element $h$ of $F(a,b)$ such that $(a^\varepsilon,h)$ forms a basis must be of type 
$h=a^nb^\eta a^m$ with $n,m\in\mathbb Z$, $\eta=\pm1$ and 
$(a^\varepsilon,a^nb^\eta a^m)$ is conjugate to 
$(a^\varepsilon,b^\eta a^{m+n})$.
\end{proof}

\section{Piecewise geodesics in hyperbolic space}

In this section we introduce piecewise geodesics and establish some basic
properties needed later on. Throughout this section all paths are paths in $\mathbb H^3$.

\medskip 
A $(N,\alpha)$-piecewise geodesic is a path that is composed of geodesic segments $[x_i,x_{i+1}]$ of length at least $N$ such that the angle $\theta_{i+1}\in[0,\pi]$ between $[x_i,x_{i+1}]$ and $[x_{i+1},x_{i+2}]$ at $x_{i+1}$ is at least $\alpha$.

	\begin{figure}[htb]
		\input{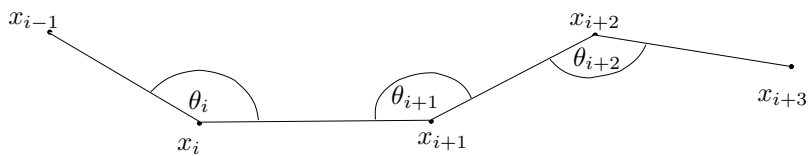}
		\caption{A section of a piecewise geodesic}
		\label{fig:pw-geodesic}
	\end{figure}

It is obvious that if $B\ge B'$ and $\alpha\ge \alpha'$ then any $(B,\alpha)$-piecewise geodesic is also a $(B',\alpha')$-piecewise geodesic.

\medskip We will need the following basic fact about piecewise geodesics; as its proof is entirely standard we merely sketch it. For definitions of quasigeodesics and local quasigeodesics and their basic properties used in the proof below the reader is referred to \cite{CDP}.

\begin{lem} \label{quasigeodesic}
 For any $\xi>0$ there exist $B_1>0$ and $\theta_0\in[0,\pi)$ such that if $B\ge B_1$ and $\alpha\in[\theta_0,\pi]$ then any bi-infinite $(B,\alpha)$-piecewise geodesic $\gamma$ is 
 $\xi$-Hausdorff-close to some geodesic $\beta$. Moreover $\gamma$ is a quasigeodesic with the same ends as~$\beta$.
 \end{lem}

\begin{proof} For any geodesic $\beta $ and $x\in \mathbb H^3$ we
  denote  the nearest point projection of $x$ to $\beta$ by $p_{\beta}(x)$.

Note first that for any $\eta>0$ there exists some angle $\theta\in[0,\pi)$
such that for any geodesic triangle ABC in $\mathbb H^3$ whose angle
at A is greater or equal to $\theta$ the sides AB and AC lie in the
$\eta$-neighborhood of BC. This is most easily seen in the Poincare disk model
by choosing A to be the center. As the angle at this vertex tends to $\pi$ the
opposite side of the triangle comes arbitrarily close the center, independently
of the length of the sides of the triangle.

	\begin{figure}[htb]
		\input{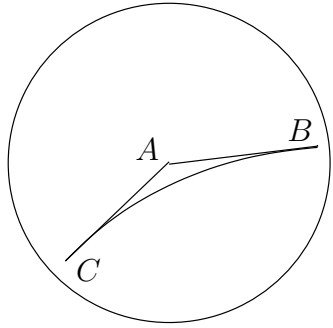}
		\caption{As the angle at A increases $d(A,BC)$ decreases}
		\label{fig:quadrilateral}
	\end{figure}

If $\theta_0$ is chosen such that the above holds for $\eta=\min(\frac{\xi}{2},\frac{1}{2})$ then it
is immediate that any $(B,\alpha)$-piecewise geodesic with $\alpha\ge\theta_0$ is a $(B,1,1)$-local
quasigeodesic. If moreover $B_0$ is chosen
sufficiently large then the local-to-global phenomenon for quasigeodesics implies that any
$(B,\alpha)$-piecewise geodesic with $B\ge B_0$ and $\alpha\ge\theta_0$ is a $(\lambda,c)$-quasigeodesic for
some fixed $\lambda\ge 1$ and $c\ge 0$.

As quasigeodics stay within bounded distance of geodesics this implies
that any $(B,\alpha)$-piecewise geodesic $\gamma$ with $B\ge B_0$ and $\alpha\ge\theta_0$  stays within bounded distance of the geodesic $\beta$ that has
the same ends in $\partial\mathbb H^3$ as $\gamma$.  This bound $C$ is uniform,
i.e. it only depends on $B_0$ and $\theta_0$.

Now observe that there exists $B_1\ge B_0$ such that for any $(B,\alpha)$-piecewise geodesic $\gamma$ with $B\ge B_1$ and $\alpha\ge\theta_0$ the midpoints $m_i$ of the segments
$[x_i,x_{i+1}]$ lie in the $\frac{\xi}{2}$-neighborhood of
$\beta$. This is true as
the quadrilaterals spanned by $x_i$, $x_{i+1}$, $p_{\beta}(x_{i+1})$ and
$p_{\beta}(x_i)$ can be assured to be arbitrarily thin provided that the
distance between $x_i$ and $x_{i+1}$ is sufficently large. This is true as
$d(x_i,p_{\beta}(x_i))$ and $d(x_{i+1},p_{\beta}(x_{i+1}))$ are bounded from
above by $C$, see Figure~\ref{fig:3}.

	\begin{figure}[htb]
		\input{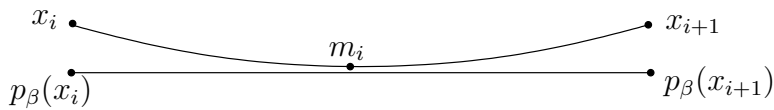}
		\caption{$m_i$ gets arbitrarily close to $\beta$ as $d(x_i,x_{i+1})$ increases}
		\label{fig:3}
	\end{figure}

Note further that the convexity of the distance function to the
geodesic $\beta $ then immediately implies that the geodesics $[m_i,m_{i+1}]$
also lie in the  $\frac{\xi}{2}$-neigborhood of $\beta$ for all $i$.

To conclude it clearly suffices to show that the piecewise geodesic
$[m_i,x_{i+1}]\cup [x_{i+1},m_{i+1}]$ lies in the
$\frac{\xi}{2}$-neighborhood of $[m_i,m_{i+1}]$. This however is true
by our choice of $\theta_0$. 

Thus we have shown that there exist $B_1>0$ and $\theta_0\in[0,\pi)$ such that
any  $(B,\alpha)$-piecewise geodesic $\gamma$ with $B\ge B_1$ and
$\alpha\ge\theta_0$ is a quasigeodesic that remains within distance
$\xi$ of the geodesic with the same ends.
\end{proof}

\section{Hyperbolic knot complements as limits of closed hyperbolic 3-manifolds}\label{secmichael}

It is a deep insight of Thurston that a cusped finite volume hyperbolic 3-manifold $M$ occurs as the geometric limit of closed hyperbolic manifolds which are topologically obtained from $M$ by Dehn fillings along increasingly complicated slopes. For definitions and details concerning algebraic and geometric convergence see 
\cite[Chapter~4]{M} and \cite[Chapter~7]{MT98}.

\smallskip Let now $M$ be the complement of a hyperbolic knot $\mathfrak k$ and let $m$ and $l$ denote the meridian/longitude 
pair for $\mathfrak k$. Let 
$\rho\co\pi_1(M)\to PSL(2,\mathbb C)$ denote the holonomy  of the complete hyperbolic structure. The image of $m$ and $l$ are   commuting parabolic elements and we can assume that their common fixed point is $\infty$ i.e.\ $Stab_\infty=\langle \rho(m),\rho(l) \rangle$ and that
\[ \rho(m) (z) = z+1 \text{ and } \rho(l)(z)= z+\tau_0\,.\]
The complex number $\tau_0\in\mathbb C\setminus\mathbb R$ is called the \emph{cusp parameter} of $\mathfrak k$.  

The deformation space of hyperbolic structures on $M$ can be holomorphically parametrised by a complex parameter $u$ in a neighborhood $U$ of $0\in\mathbb{C}$. Details about the deformation space and Thurston's hyperbolic Dehn filling theorem can be found in the Appendix~B of \cite{BP01}. See also \cite[Chapter~5]{Th}.

The following facts can be found in \cite[B.1.2]{BP01}:
there is 
an analytic family 
$\rho_u$, $u\in U$, of representations 
$\rho_u\co\pi_1 (M) \to PSL(2,\mathbb C)$ and an analytic function $v=v(u)$ such that $u$ and $v$ are the complex translation length of $\rho_u(m)$ and $\rho_u(l)$ respectively and $v(0)=0$. The function $\tau(u)=v(u)/u$ is analytic and $\tau(u) = \tau(0) + O(|u|^2)$ where $\tau(0)=\tau_0$ is the cusp parameter. 
 For $u\in U$ the generalized \emph{Dehn filling coefficient} of the cusp is
 the defined to be the element of $\mathbb{R}^2 \cup\infty\cong S^2$ defined by
\[ \begin{cases} \infty &\text{ if $u=0$}\\
(p,q)\  \hbox{s.t.}\  u p + v q = 2\pi i &\text{ if $u\neq0$.}
\end{cases}\]
By Thurston's Dehn filling theorem and the definition of $U$, the map from the points of $U$ to their generalized Dehn filling coefficients is a homeomorphism from $U$ to a neighborhood of $\infty$ in $\mathbb{R}^2 \cup \{ \infty\}$.

The representation $\rho_0$ is the holonomy of the complete hyperbolic structure of $M$. If $u\neq0$ then
the representation $\rho_u$ is the holonomy of a non complete hyperbolic structure $M_u$ on $M$ and the metric completiton of $M_u$ is described by the Dehn filling parameters 
(see \cite[B.1]{BP01}). We are only interested in the case that $p$ and $q$
are coprime integers. Then $\rho_u$ factors through $\pi_1 (M(p/q))$ 
and the metric completion of $M_u$ is homeomorphic to 
$M(p/q)$. Here and in the sequel $M(p/q)$ denotes the manifold obtained from $M$ by Dehn filling along the slope $p m+ q l$. By Mostow-Prasad rigidity the  faithful discrete representations of $\pi_1(M)$ and $\pi_1(M(p/q))$ in $PSL(2,\mathbb C)$ are unique up to conjugation.

For a large enough integer $n$, the coordinate pair $(1,n)$ must lie in the homeomorphic image of $U$ in $\mathbb{R}^2 \cup \{ \infty \}$ so there is a $u_n \in U$ with Dehn filling coefficient $(1,n)$.  Furthermore $u_n \to 0$ as $n \to \infty$.  Hence $v_n := v(u_n) \to 0$ and $\tau_n:=\tau(u_n) \to \tau_0$ as $n \to \infty$.  So, as described above, the metric completion of $M_{u_n}$ is the hyperbolic manifold $M_n:=M(1/n)$ and the holonomy for $M_n$ is determined by $\rho_n:=\rho_{u_n}$.    The manifold $M_n$ contains a new geodesic, the core of the filling torus, which is isotopic to the image of $l$ in $M_n$.  

On the peripheral subgroup $\langle m, l \rangle < \pi_1 (M)$ the representation is given by:


\[ \rho_n(m)(z)= e^{u_n} z +1 \text{ and } \rho_n(l)(z)= e^{v_n} z +\frac{e^{v_n}-1}{e^{u_n}-1}\]
(see \cite[B.1]{BP01}).
By conjugation of $\rho_n$ by the parabolic transformation
$A_n$ given by $$A_n(z) = z +\frac{\tau_n}{1-e^{v_n}} - \frac{1}{1-e^{u_n}}$$
we can assume that $\rho_n(m) = W_n$ and $\rho_n(l) = V_n$ where
\[ W_n(z)= e^{u_n} z +\tau_n\frac{e^{u_n}-1}{e^{v_n}-1} \text{ and } V_n(z)= e^{v_n} z +\tau_n\,.\]
Note that $A_n$ converges to $A$ given by $A(z)=z+ (\tau_0-1)/2$.
Note also that the Dehn Surgery Theorem (see \cite[Chapter~5]{Th} and
\cite{PePo}) implies that the sequence of groups 
$\{\rho_n(\pi_1(M))\}$ converges geometrically to $\rho(\pi_1(M))$.

A simple calculation shows that 
\[ V_n^k(z) = e^{kv_n} z +\frac{e^{kv_n}-1}{e^{v_n}-1}\tau_n\]
and hence $V_n^{-n} = W_n$ since $u_n+nv_n=2\pi i$ and therefore $e^{-nv_n}=e^{u_n}$.
Moreover the fixed points of the loxodromic transformation $V_n$ are $\infty$
and $\tau_n/(1-e^{v_n})$. The above facts allow us to conclude as in the
discussion in \cite[4.9]{M} that the following hold:
\begin{enumerate}
\item The elements  $\rho_n(l)$ are loxodromic isometries. The fixed points of $\rho_n(l)$ converge to the fixed point $\infty$ of the parabolic subgroup $P=\rho(\pi_1(\partial M))$.

\item $\rho_n(l)$ converges uniformly on compact sets to $ \rho(l)$ and 
$\rho_n(m)=\rho_n(l)^{-n}$ converges to $\rho(m)$.

\item Furthermore the sequence of subgroups 
$\langle \rho_n(l)\rangle$ generated by the core of the filling solid torus converges geometrically to the peripheral subgroup 
$P$.
\end{enumerate}

In the sequel we shall use the following convention: we shall identify $\pi_1(M)$ with the image $\rho(\pi_1(M))\subset PSL(2,\mathbb C)$ and for each $g\in \pi_1(M)$ we  write $g_n = \rho_n(g)$ for short.
We shall denote by $\gamma_n\subset M_n$ the new geodesic i.e.\ the core of the filling solid torus.

The following proposition gives some more information about the geometry of the limiting process. Note that the translation lengths of an element $g$ on some $g$-invariant subset $Y$ of $\mathbb H^3$ is defined to be $$|g|_Y:=\underset{y\in Y}{\inf}d_{\mathbb H^3}(y,gy).$$ In particular the translation length of a parabolic element $g$ on a $g$-invariant horosphere $S$ is measured with respect to the metric of $\mathbb H^3$ rather than the Euclidean path metric of $H$. We will need the following lemma, see \cite[Sec.~9, Lemma~2]{Mey87}).

\begin{lem}\label{mey} Let $g\in\hbox{Isom}(\mathbb H^3)$ be a loxodromic isometry with complex translation length $a+ib$ and let $N_r=N_r(A_g)$ be the $r$ neighborhood of the axis of $g$. Then $$\cosh(|g|_{\partial N_r})=\cosh(a)+\sinh^2(r)(\cosh(a)-\cos(b))=$$ $$\cosh(a)+\sinh^2(r)|\cosh(a+ib)-1|.$$
\end{lem}

Note that the two formulae in Lemma~\ref{mey} are equivalent via the identity \begin{equation}\label{eq:identity}|\cosh(a+ib)-1|=\cosh(a)-\cos(b).\end{equation}

\begin{prop}\label{michael}
For any horoball $H$ at $\infty$ there exists  a sequence 
$(r_n)_{n\in\mathbb N}$ of real numbers such that the following hold where $N_n:=N_{r_n}(A_{l_n})$ is the $r_n$-neighborhood of the axis $A_{l_n}$ of $l_n$ in $\mathbb H^3$.
\begin{enumerate}
\item For any fixed $k,N$ we have $\lim_{n\to\infty}\left(|l_n^{N-k\cdot n}|_{\partial N_n}\right)=|l^N m^k|_{\partial H}$.

\item If $N\in\mathbb N$  such that $|l^Nm^k|_{\partial H}>C>0$ for all $k\in \mathbb Z$, then for any $\eta>0$ there exists some $n'$ such that  $|l_n^{N-k\cdot n}|_{\partial N_n}\ge C -\eta$ for all $n\geq n'$ and all 
$k\in\mathbb Z$.
\end{enumerate}

Moreover if $(g_n)$ is a sequence of elements with $g_n\in\rho_n(\pi_1(M_n))$ that converges to a hyperbolic element $g\in \rho(\pi_1(M))$ and $C>0$ then $H$ can be chosen such that the following hold:
\begin{enumerate}\alph{enumi}
\item[(a)]\label{(a)} For sufficiently large $n$  we have
  $d(N_n,h_nN_n)\ge C$  for all 
  $h_n\in \rho_n(\pi_1(M_n))-\langle l_n\rangle$.
\item[(b)]\label{(b)} For sufficiently large $n$ the $C$-neighborhood of 
the geodesic segment $[x_n,y_n]$ between $N_n$ and $g_n N_n$ does not intersect any translates of $N_n$ except $N_n$ and $g_n N_n$.
\end{enumerate}
\end{prop}

\begin{proof}
Let $H = \{ (x,t)\in \mathbb C\times\mathbb R^+= \mathbb H^3 \mid t\geq t_0\}$ be a closed horoball in the upper half space model of the $\mathbb H^3$.
For a complex number $z$ we will denote by $\Re(z)$ and $\Im(z)$ the real and the imaginary part of $z$.

The translation length $|l|_{\partial H}$ is given by
\begin{equation}\label{eq:horo}\cosh(|l|_{\partial H}) = 1 +\frac{ |\tau_0|^2}{2t^2_0}\;.\end{equation}
We now define $r_n := 0$ if $|\Re(v_n)| \ge |l|_{\partial H}$.
If $|\Re(v_n)| < |l|_{\partial H}$ then we define  $r_n$ to be the unique positive real number satisfying 
\begin{equation}\label{eq:rn}
 \sinh^2(r_n)=
\frac { \cosh(|l|_{\partial H}) -\cosh( \Re(v_n) )}{|\cosh(v_n) -1 |}\,.
\end{equation} 
Note that since $v_n \to 0$ as $n \to \infty$ and $|l|_{\partial H}$ is a fixed positive number, eventually equation~\eqref{eq:rn} always holds and $r_n \to \infty$.
This definition implies that 
$|l_n|_{\partial N_n} = |l|_{\partial H}$ if $|\Re(v_n)| < |l|_{\partial H}$. Indeed this follows from Lemma~\ref{mey} and the fact that $v_n$  is the complex translation length of $l_n$.


Fix now $N$ and $k$. The complex translation length of $l_n^{N-kn}$
is $(N-kn)v_n$ and the equation $u_n + n v_n = 2\pi i $ implies
\[ (N-kn)v_n \equiv N v_n + k u_n \mod 2\pi i \,. \]
Hence, by Lemma~\ref{mey}, the translation length $|l_n^{N-kn}|_{\partial N_n}$ is given by

\begin{align}\label{eq:Translength}
\cosh(|l_n^{N-kn}|_{\partial N_n}) &=  \cosh( \Re(N v_n + k u_n) )  \\ \notag
& \quad + \sinh^2(r_n) \cdot | \cosh(N v_n + k u_n) -1|\,.
\end{align}

Now it follows from the low order asymptotics of the hyperbolic cosine and its Taylor expansion that
\[ \frac{| \cosh(N v_n + k u_n) -1| } { | \cosh(v_n) -1|} =
\frac{ | N \tau_n + k |^2}{|\tau_n|^2} \big(1+O(|v_n|^2 )\big)\,.\]
Moreover equation~(\ref{eq:rn}) implies
\[\lim_{n\to\infty}| \cosh(v_n) -1|\cdot \sinh^2(r_n) = \cosh(|l|_{\partial H})-1\,.\]
The last two equations together with equation~(\ref{eq:Translength})
imply:
\begin{align*}
\cosh(|l_n^{N-kn}|_{\partial N_n})
&= \cosh( \Re(N v_n + k u_n) )  \\ \notag
& \quad + \sinh^2(r_n) \cdot | \cosh(v_n) -1|
\frac{| \cosh(N v_n + k u_n) -1|}{| \cosh(v_n) -1|}\\
&\stackrel{n\to\infty}{\longrightarrow}
1 + (\cosh(|l|_{\partial H})-1)\frac{|N \tau_0 + k|^2}{|\tau_0|^2}\\
&= 1 + \frac{ | N \tau_0 + k |^2}{2t^2_0} \\
&= 
\cosh(|l^{N}m^k|_{\partial H})\;.
\end{align*}
This proves the first point.

In order to prove the second point we will make use of the following limits
\begin{equation}\label{eq:lim}
\lim_{n\to\infty}\frac{|\Re(v_n)|}{|v_n|} = 0 \text{ and }
\lim_{n\to\infty} \frac{|\Re(u_n)|}{|v_n|} = \frac{|\Im(\tau_0)|}{|\tau_0|^2}>0\,.
\end{equation}
and 
\begin{equation}\lim_{n\to\infty} |v_n|  \cosh(r_n) = \frac{|\tau_0|}{t_0} \ \end{equation} 
The first two follow easily using the facts $u_n+nv_n=2\pi i$, $v_n=u_n \tau_n$, and $\lim_{n \to \infty} \tau_n = \tau_0$.  To verify the second one observe first that ~(\ref{eq:rn}), the identity $\sinh^2(z)=\cosh^2(z)-1$ and multiplication by $|v_n|^2$ imply that
$$ |v_n|^2\cosh^2(r_n)=|v_n|^2\frac{\cosh(|l|_{\partial H})-\cosh(\Re(v_n))+|\cosh(v_n)-1|}{|\cosh(v_n)-1|}$$
applying (\ref{eq:identity}) and (\ref{eq:horo}) then gives 
\begin{align*} 
|v_n|^2\cosh^2(r_n) &=|v_n|^2\frac{\cosh(|l|_{\partial H})-\cos(\Im(v_n))}{|\cosh(v_n)-1|}\\ &=|v_n|^2\frac{1+\frac{|\tau_0|^2}{2t_0^2}-\cos(\Im(v_n))}{|\cosh(v_n)-1|}.\end{align*}

Now as $\lim\limits_{z\to 0}\frac{|z|^2}{|\cosh(z)-1|}=2$ this implies that $\lim\limits_{n\to\infty} |v_n|^2\cosh^2(r_n)=\frac{|\tau_0|^2}{t_0^2}$ which clearly proves the claim.

Suppose now that $N\in\mathbb N$ satisfies 
$|l^Nm^k|_{\partial H}>C>0$ for all $k\in \mathbb Z$ and let 
$\eta> 0$ be given. 
We choose $n''=n''(N,C)$ such that for all
$n\geq n''$ the following holds:
\[ \bigg| N \frac{\Re(v_n)}{|v_n|}\bigg| < 1,\quad
 |v_n|\cdot \cosh(r_n)  \geq \frac{|\tau_0|}{2t_0}\]
and
\[  \frac{|\Re(u_n)|}{|v_n|} > \frac 1 2 \frac{|\Im(\tau_0)|}{|\tau_0|^2}\,.\]

In order to prove the second point we start again with formula~(\ref{eq:Translength}):
\begin{align*}
\cosh(|l_n^{N-kn}|_{\partial N_n}) 
&= \cosh\big( \Re(N v_n + k u_n) \big)\cosh^2(r_n)\\ &\qquad  -   
\cos\big(\Im(N v_n + k u_n)\big) \sinh^2(r_n)\\
&> \big( \cosh( \Re(N v_n + k u_n) ) - 1\big)\cosh^2(r_n)\\
&\geq \frac 1 2 \big(\Re(N v_n + k u_n)\big)^2\cosh^2(r_n)\\
&= \frac 1 2 \big(\Re(N \frac{v_n}{|v_n|} + k \frac{u_n}{|v_n|})\big)^2 
(|v_n|\cosh(r_n))^2
\intertext{and for $n\geq n''$ we obtain }
\cosh(|l_n^{N-kn}|_{\partial N_n})&>  \frac{|\tau_0|^2}{8t^2_0} \cdot \big| |k \frac{\Re(u_n)}{|v_n|}| - | N \frac{\Re(v_n)}{|v_n|} | \big|^2. 
\end{align*}
Hence there exists a constant $C'$ such that $|k|\geq C'$ implies
\[|l_n^{N-kn}|_{\partial N_n}\geq C\,.\]
Since there are only finitely many $k\in\mathbb Z$ such that
$|k|< C'$ it follows from the first part of the
proposition that we can find $n'''$ such that for all $n\geq n'''$ and all
$|k|< C'$ the equation
\[ \big| l_n^{N-kn} \big|_{\partial N_n} \geq C-\eta\]
holds. The second point follows for $n'=\max(n'',n''')$.

In order to prove (a) and (b) we  shall use some results of Meyerhoff.
By \cite[Sec.~3\&9]{Mey87} and the definition of $r_n$, we know that if we choose $H$ such that $\varepsilon=|l|_{\partial H}$ is sufficiently small then the sets $$N_n=N_{r_n}(A_{l_n})=\{p \in \mathbb H^3 \mid d_{\mathbb H^3}(p, l_n(p) )\leq \varepsilon\}$$
have the property that $h_n(N_n)\cap N_n = \emptyset$ for all $h_n \in \rho_n(\pi_1(M_n))-\langle l_n \rangle$

By further reducing $\varepsilon$ we can further assume that $d(N_n,h_nN_n)\ge C$  for all $h_n \in \rho_n(\pi_1(M_n))-\langle l_n \rangle$  as the radii $r_n$ decrease uniformly as $\varepsilon$ decreases as follows from (\ref{eq:rn}). This proves (a).


We now put $\tilde H=GH$ and $\tilde N_n=G_nN_n$ where
$G=\rho(\pi_1(M))$ and $G_n=\rho_n(\pi_1(M))$. It now follows from the geometric convergence of $(G_n)$ to $G$ that for all 
$\eta>0$ and compact $K\subset\mathbb H^3$ there exists $n'$ such that for all $n\ge n'$  
$$d_H(K\cap \tilde H,K\cap \tilde N_n)\le \eta.$$ 
Here $d_H$ denotes the Hausdorff distance. Moreover for any $h\in G$  and any sequence $h_n\to h$ we have $d_H(K\cap hH,K\cap h_nN_n)\le \eta$ for $n$ sufficiently large.

Let $[x,y]$ denote the geodesic segment between $H$ and $gH$ and let  $[x_n,y_n]$ denote  the geodesic segment between $N_n$ and $g_n N_n$. 

Note that there can be only finitely many translates
$h H$, $h\in\pi_1(M)$, $h\not\in P\cup gP$, such that the intersection of $h H$ and the $(C+1)$-neighborhood of $[x,y]$ is non-empty. This follows from the fact that the $(C+1)$-neighborhood of $[x,y]$ is compact and that the translates of $H$ are disjoint. Thus after decreasing $\varepsilon$ we can assume that the $(C+1)$-neighborhood of $[x,y]$ does not intersect $hH$ for $h\not\in P\cup gP$. 

Let now $K$ be the $(C+1)$-neighborhood of $[x,y]$.
The above remark applies to $K$. Hence we know that for sufficiently large $n$ the Hausdorff distances between $H\cap K$ and $N_n\cap K$ and between $K\cap gH$ and $K\cap g_nN_n$ are arbitrarily small. This implies that the segments $[x_n,y_n]$ converge to $[x,y]$ thus $d_H([x,y],[x_n,y_n])\le\frac{1}{2}$ for large $n$. We also see that the $(C+\frac{1}{2})$-neighborhood of $[x,y]$ does not meet any translates of $N_n$ besides $N_n$ and $g_nN_n$ for large $n$. Thus we have shown that the $C$-neighborhood of $[x_n,y_n]$ does not meet any translates of $N_n$ besides $N_n$ and $g_nN_n$. This proves (b).\end{proof}

\begin{lem}\label{lemmawinkel} For any $\beta\in(0,\pi/2)$ there exist $\kappa(\beta)$ and $r(\beta)$ such that if  $\gamma$ is a geodesic in $\mathbb H^3$, $r\ge r(\beta)$ and $x,y\in\partial(N_r(\gamma))$ such that $d(x,y)\ge \kappa(\beta)$ then the angles enclosed by the geodesic segment between $x$ and $y$ and 
$\partial(N_r(\gamma))$ are at least $\beta$.
\end{lem}
\begin{proof}
The proof of the lemma is by calculation, we follow the setup of \cite[Section~2]{GMM}. We perform all calculations in the Klein hyperboloid model of $\mathbb H^3$. In this model, $\mathbb H^3$ is the component of the hypersurface
\[ 
\{(x_0,x_1,x_2,x_3)\in\mathbb R^4\mid -x_0^2+x_1^2+x_2^2+x_3^2 = -1\}
\] with $x_0>0$. For $x=(x_0,x_1,x_2,x_3)$ and $y=(y_0,y_1,y_2,y_3)$ we will denote by $\langle x,y\rangle$ the Minkowski inner product
\[ \langle x,y\rangle = - x_0y_0 + x_1 y_1 + x_2 y_2 +x_3 y_3\,.\]
Recall that for each $x\in\mathbb H^3$ the restriction of the 
Minkowski inner product on the tangent space 
$T_{x} \mathbb H^3$ is positive defined. In the sequel we will use this metric on the tangent space. 

We assume that $\gamma$ is the intersection of $\mathbb H^3$ with the plane $\{ x_1=x_3=0\}$. Let $g$ be the loxodromic motion along $\gamma$ with complex length $\delta+i\phi$. The isometry $g$ is represented by the matrix 
\[M_g:=
\begin{pmatrix}
\cosh(\delta) & 0 & \sinh(\delta) &0\\
0 & \cos(\phi) & 0 & -\sin(\phi) \\
\sinh(\delta) & 0 & \cosh(\delta) &0\\
0 & \sin(\phi) & 0 & \cos(\phi)
\end{pmatrix}\,.
\]

Let $r>0$. Fix the point $x=(\cosh(r),\sinh(r),0,0)$. Let  
$\overrightarrow{n}\in T_{x} \mathbb H^3$
be the unit length inward normal vector to $\partial(N_r(\gamma))$. 
Thus $$\overrightarrow{n}=(-\sinh(r), -\cosh(r),0,0).$$
For every $y\in\mathbb H^3$ the unit vector $\overrightarrow{m}_y\in T_x \mathbb H^3$ pointing into the direction of $y$ is given by
\[ \overrightarrow m_y =\frac{y+ \langle y,x \rangle x }
{(\langle y,x \rangle^2 -1)^{1/2} }\,.\]

Now suppose that $y \in \partial (N_r(\gamma))$.  Hence there is some $\delta\in \mathbb{R}$ and $\phi \in [0,2\pi)$ such that $M_g x=y$. Denote by $\beta(y)$, $0\leq \beta(y) \leq \pi/2$, the angle enclosed by the geodesic segment $[x,y]$ and 
$\partial(N_r(\gamma))$. We need to show that $\beta(y)$ is arbitrarily close to $\pi/2$ 
provided that $r$ and $d(x,y)$ are sufficiently large.

 Recall that if $p$ and $q$ are points of $\mathbb{H}^3$ in the hyperboloid model then the hyperbolic distance between them is given by the formula 
 \[\cosh \big(d(p,  q)\big) = - \langle p, q\rangle.\]  
 We have 
 \[\sin(\beta)) = \cos \left( \frac{\pi}{2}-\beta(y) \right)  = \langle \overrightarrow{m}_y, \overrightarrow{n} \rangle.\]  
 It now follows from the definitions of $x$, $M_g$, 
 $\overrightarrow{m}_y$, and $\overrightarrow{n}$ that 
 \[ \langle \overrightarrow{m}_y, \overrightarrow{n} \rangle =  
 \frac{\cosh(r) \sinh(r)}{\sinh(d(x,y))}(\cosh(\delta)-\cos(\phi)).\]  Moreover, using the above distance formula, we have 
 \[ \cosh(d(x,y)) = \cosh(\delta)\cosh^2(r)-\cos(\phi)\sinh^2(r)\] 
 and hence
\[\sin(\beta(y)) = \tanh (r) \frac{\cosh(d(x,y))-\cos(\phi)}{\sinh(d(x,y))}\,.\]
For a fixed distance $d=d(x,y)$ the angle $\beta(y)$ becomes minimal if $\phi=0$.  
Therefore
\[\sin(\beta(y))\geq  
\tanh (r) \frac{\cosh(d(x,y))-1}{\sinh(d(x,y))}=
\tanh (r) \tanh\bigg(\frac{d(x,y)}{2}\bigg)\,.\]

Now let  $\beta$, $0<\beta<\pi/2$, be given. 
We choose $r(\beta)>0$ such that $\sin(\beta)<\sin(\beta) \coth(r(\beta))=q<1$ and $\kappa(\beta)$ such that
$q=\tanh(\kappa(\beta)/2)$. Hence for $r\geq r(\beta)$ and for $y\in\partial N_r$ such that
$d(x,y)\geq\kappa(\beta)$ we  obtain
\begin{align*}
\sin(\beta(y)) &\geq \tanh (r) \tanh\big(d(x,y)/2\big)\\
&\geq \tanh (r(\beta)) \tanh\big(\kappa(\beta)/2\big) = \sin(\beta)\,.
\end{align*}
Therefore we have for all $r\geq r(\beta)$ and all $y\in\partial N_r$ such that $d(x,y) \geq \kappa(\beta)$ that $\beta(y)\geq\beta$. This proves the Lemma.
\end{proof}

\section{Generating pairs of 2-bridge knot groups}

In this section we prove that hyperbolic 2-bridge knot groups have infinitely many Nielsen classes of generating pairs. Moreover we prove that there exist closed hyperbolic 3-manifolds that have arbitrarily many Nielsen classes of generating pairs. Those manifolds are obtained by Dehn surgery on $S^3$ at 2-bridge knots.

\medskip  The infinity of Nielsen classes of generating pairs of fundamental groups of  Seifert fibered 2-bridge knot spaces has been known for a long time. For the trefoil knot this is due to Dunwoody and Pietrowsky \cite{DP} and the general case is due to Zieschang \cite{Z2} who in fact gives a complete classification of Nielsen classes of generating pairs. 

\medskip
Let $M$ be the exterior of a hyperbolic 2-bridge knot $\mathfrak k$. Choose $m,l\in\pi_1(M)$ such that $m$ represents the meridian, that $l$ represents the longitude and that $\langle m,l\rangle\cong \mathbb Z^2$ is a maximal peripheral subgroup. Inspecting the Wirtinger presentation shows that 2-bridge knot groups are generated by two meridional elements, i.e. that there exists some $g$ such that $\pi_1(M)=\langle m,gmg^{-1}\rangle$.

\medskip As $(gl^N)\cdot m\cdot (gl^N)^{-1}=g\cdot l^Nml^{-N}\cdot g^{-1}=gmg^{-1}$ it follows that $$P^N:=(m,gl^N)$$ is a generating pair for $\pi_1(M)$ for all $N\in\mathbb Z$. As in Section~\ref{secmichael} we put $M_n:=M(1/n)$ and denote the image of an element $h\in \pi_1(M)$ in $\pi_1(M_n)$ by $h_n$. In particular $P^N_n=(m_n,g_nl_n^N)$ is a generating pair of $\pi_1(M_n)$ for all $n\in\mathbb N$,

\medskip The following theorem is the main theorem of this article.

\begin{thm}\label{main} 
 There exists $N_0\in\mathbb N$ such that for any $N,N'\ge N_0$ there exists some $n_0$ such that for  $n\ge n_0$ the generating pairs $P_n^N$ and $P_n^{N'}$ of $\pi_1(M_n)$ are not Nielsen equivalent.
\end{thm}

Note that for all $n$, $N$ and $N'$ the generating pairs $P_n^N$ and $P_n^{N'}$ have the same commutator as 
\[ [m_n,g_nl_n^N]=m_n \cdot g_n l_n^N\cdot m_n^{-1} \cdot l_n^{-N}g_n^{-1}=m_ng_nm_n^{-1}g_n^{-1}= [m_n , g_n].\]
Thus we cannot apply Proposition~\ref{Nielsen} to distinguish the Nielsen equivalence classes of $P_n^N$ and $P_n^{N'}$. We get the two results stated in the introduction as immediate corollaries.

\begin{cor} For any $n$ there exists a closed hyperbolic 3-manifold $M$ such that $\pi_1(M)$ has at least $n$ distinct Nielsen classes of generating pairs.
\end{cor}

Since Nielsen-equivalent tuples cannot become non-equivalent in a quotient group we also get the following.

\begin{cor} 
Let $\mathfrak k$ be a hyperbolic $2$-bridge knot with knot exterior $M$. Then $\pi_1(M)$ has infinitely many Nielsen classes of generating pairs. 
\end{cor}

We will prove two lemmas before  giving the proof of Theorem~\ref{main}. We use the same  notations as in Section~\ref{secmichael}.

\begin{lem}\label{lem1} There exists $N_1$ such that for any $N\ge N_1$ there exists some $n_1\in\mathbb N$ so that if $n\ge n_1$, $\varepsilon,\eta\in\{-1,1\}$, and $w$ is  a positive word  in $m_n^{\eta}$ and $(g_nl_n^N)^{\varepsilon}$, then $w$ represents an element of $G_n$ which is conjugate in $G_n$ to $m_n^{\pm 1}$ if and only if $w$ consists of the single letter $m_n^{\eta}$.
\end{lem}

\begin{proof} The ``if'' direction is trivial. To prove the ``only if'' direction, we prove its contrapositive by showing that if $w$ is not a power of $m_n^{\eta}$ then $w$ is not conjugate in $\hbox{Isom}(\mathbb H^3)$ to $m_n^{\pm 1}$. The contrapositive then follows since a loxodromic element is never conjugate to a proper power of itself. We will prove this in the special case that $\eta=\varepsilon=1$, the other cases are analogous.

Take $\xi=1/2$. Using this choice of $\xi$, Lemma~\ref{quasigeodesic} gives $B_1>0$ and $\theta_0\in[0,\pi)$. Choose $B\ge \max(100,B_1)$ and $\alpha\in [\theta_0,\pi)\cap (\pi/2,\pi)$. Then the conclusion of Lemma~\ref{quasigeodesic} holds for our choice of $\xi$ and any bi-infinite $(B,\alpha)$-piecewise geodesic. Furthermore choose $\kappa:=\kappa(\alpha-\pi/2)$ and 
$r:=r(\alpha-\pi/2)$ as in Lemma~\ref{lemmawinkel}. 

\smallskip 

By Proposition~\ref{michael} there is a horoball $H$ centered at the fixed point of $\langle m,l\rangle$ and a sequence $(r_n)$ of positive real numbers so that the four conclusions of the proposition hold for $C=B$ and the convergent sequence $(g_n)$. In particular there exist $\hat n$ such that $d(N_n,g_nN_n)\ge B$ for $n\ge \hat n$. Now let $x\in H$, $y\in gH$ and $x_n\in N_n$, $y_n\in g_nN_n$ for all $n\in\mathbb N$ be such  that $[x,y]$ is the geodesic segment between $H$ and $gH$ and that $[x_n,y_n]$ is the geodesic segment between $N_n$ and $g_nN_n$ for all $n\in\mathbb N$.

	\begin{figure}[htb]
		\input{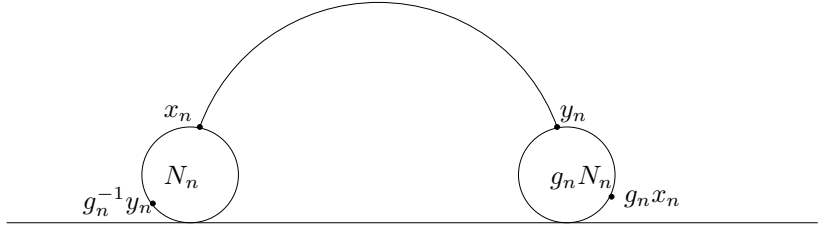}
		\caption{The horoball approximation $N_n$ and its translate by $g_n$}
		\label{fig:horoball}
	\end{figure}

\smallskip Choose $t$ such that $d(g_nx_n,y_n)\le t$ for all $n$. Such $t$ clearly exists as the segments $[x_n,y_n]$ converge to the segment $[x,y]$ and $g_n$ to $g$. Choose further $N_1$ such that $|l^Nm^k|_{\partial H}\ge B+t+\kappa+1$ for all $N\ge N_1$ and $k\in\mathbb Z$. Such $N_1$ clearly exists as for any constant $K$ there are only finitely many elements $h\in\langle m,l\rangle$ such that $|h|_{\partial H}\le K$.

Now fix $N\ge N_1$. Recall that $m_n=l_n^{-n}$. It now follows from Proposition~\ref{michael}~(2) (by chosing $\eta$ sufficiently small) that there exists some $\tilde n\ge \hat n$ such that $|l_n^Nm_n^k|_{\partial N_n}\ge B+t+\kappa$ for all $k\in\mathbb Z$ and $n\ge \tilde n$.

\smallskip Now choose $n_1\ge \tilde n$ such that $|m_n|_{\mathbb H^3}<1$ and $r_n>r(\alpha-\pi/2)$ for all $n\ge n_1$; this is clearly possible as the element $m_n$ converges to the parabolic element $m$ and $r_n$ tends to infinity as $n$ tends to infinity. 
Fix $n\ge n_1$ and suppose that $w$ is not a power of $m_n$. We will show that the translation length of $w$ is at least 198. This is strictly larger than the translation length of $m_n$, hence this will prove the lemma. Since $w$ is not a proper power of $m_n$ we may conjugate $w$ to assume that it is of the form

\[
(g_nl_n^N)m_n^{b_1}\cdot\ldots\cdot (g_nl_n^N)m_n^{b_s}
\] 
with $b_i\ge 0$ for $1\le i\le s$. Thus $w$ can be rewritten as a product $$(g_np_1)\cdot\ldots\cdot (g_np_s)$$ where 
$p_i=l_n^Nm_n^{b_i}$ for $1\le i\le s$. By the above choices the translation length of all $p_i$ on $\partial N_n$ is at least $B+t+\kappa$. 

\smallskip We now construct a $w$-invariant bi-infinite $(B,\alpha)$-piecewise geodesic $\gamma_w$ containing $x_n$. We first construct a $(B,\alpha)$-piecewise geodesic $\gamma_0$ from $x_n$ to $wx_n$ and put $\gamma_i=w^i\gamma_0$. We then put $$\gamma_w:=\ldots \gamma_{-2}\cdot \gamma_{-1}\cdot \gamma_0\cdot \gamma_1\cdot \gamma_2\ldots $$ which clearly implies the $w$-invariance of $\gamma_w$. The fact that $\gamma_w$ is also a $(B,\alpha)$-piecewise geodesic follows immediately from the construction.

\smallskip  For $1\le i\le s$ put $w_i=(g_np_1)\cdot\ldots \cdot (g_np_i)$,  $x_n^i=w_ix_n$ and  $y_n^i=w_iy_n$. We then put  $$\gamma_0:=[x_n,y_n]\cdot [y_n,x^1_n]\cdot [x^1_n,y^1_n]\cdot [y^1_n,x^2_n]\cdot\ldots\cdot [x^{s-1}_n,y^{s-1}_n]\cdot [y^{s-1}_n,x^s_n=wx_n].$$ 

	\begin{figure}[htb]
		\input{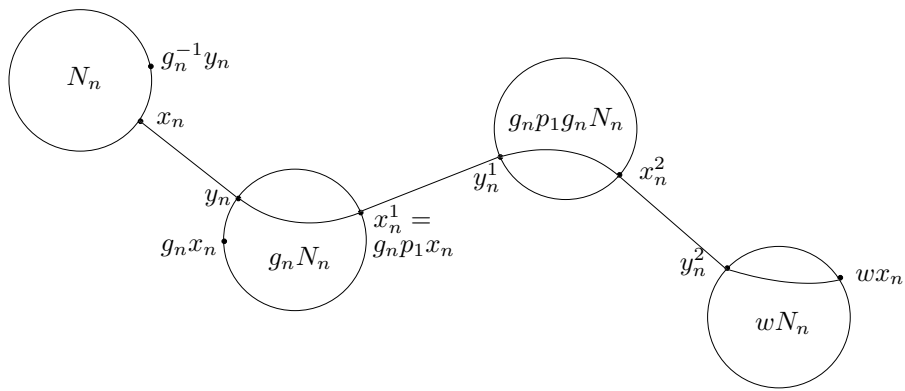}
		\caption{The piecewise geodesic $\gamma_0$ for $w=(g_np_1)(g_np_2)(g_np_3)$}
		\label{fig:gamma_0}
	\end{figure}

Now $x^i_n\in\partial (w_iN_n)=w_i\partial N_n$ and $y^i_n\in\partial (w_{i+1}N_n)$.  The segments $[x_n^i,y_n^i]$ are of length at least $B$ by assumption and are perpendicular to the respective translates of $\partial N_n$. Note further that the segments $[y_n^i,x_n^{i+1}]$ are of length at least $B+ \kappa$. Indeed this follows from the triangle inequality and the fact that $$d(w_{i}g_nx_n,y_n^i)=d(w_{i}g_nx_n,w_iy_n)=d(g_nx_n,y_n)\le t$$ and $$d(x_n^{i+1},w_{i}g_nx_n)=d(w_{i+1}x_n,w_ig_nx_n)=d(w_ig_np_{i+1}x_n,w_ig_nx_n)=$$ $$=d(p_{i+1}x_n,x_n)\ge B+t+\kappa.$$ It follows from the choice of $\kappa$ and Lemma~\ref{lemmawinkel} that the segments $[y_n^i,x_n^{i+1}]$ enclose angles greater or equal than $\alpha-\pi/2$ with the respective translates of $\partial N_n$. This proves that $\gamma_i$ and therefore $\gamma_w$ is a $(B,\alpha)$-piecewise geodesic. 

By Lemma~\ref{quasigeodesic} $\gamma_w$ is a quasigeodesic that lies in the $\frac{1}{2}$-neighborhood of the geodesic $\beta$ that has the same ends. Clearly $\beta$ is invariant under the action of $w$, thus we must have $\beta=A_w$ where $A_w$ is the axis of~$w$.

We argue that the translation length of $w$ must be at least 198. Recall that $\gamma_0$ is a $(B,\alpha)$-piecewise geodesic consisting of at least $2$ segments of length at least $B\ge 100$. As $\gamma_0$ lies in the $\frac{1}{2}$-neighborhood of $\beta$ it follows that each of the geodesic segments projects under $p_{\beta}$ to a geodesic segment of length at least $99$. Thus $$d(p_{\beta}(x_n),p_{\beta}(wx_n))=d(p_{\beta}(x_n),wp_{\beta}(x_n))\ge 2\cdot 99=198.$$ Now $d(p_{\beta}(x_n),wp_{\beta}(x_n))$ is the translation length of $w$, so the lemma is proven. 
\end{proof}

\begin{lem}\label{lem2} There exists $N_2$ such that for all $N\neq N'\ge N_2$ there exists $n_2\in\mathbb N$ such that for $n\ge n_2$ and $k\in\mathbb Z$ the elements $g_nl_n^N$ and 
$(g_nl_n^{N'+nk})^\varepsilon$ are not conjugate in $\pi_1(M_n)$ for $\varepsilon\in\{\pm 1\}$.
\end{lem}

\begin{proof} The proof is similar to the proof of Lemma~\ref{lem1}. Choose the constants $B$, $\alpha$, $\kappa$, $t$, $N_1$, the segments $[x_n,y_n]$ and the horoball $H$ as in the proof of Lemma~\ref{lem1} and put $N_2:=N_1$. Choose $N\neq N'\ge N_2$. 
Now by choosing $n_2>N+N'$ sufficiently large it follows as before that $$|l_n^{N+kn}|_{\partial N_n},|l_n^{N'+kn}|_{\partial N_n}\ge B+t+\kappa$$
for all $n\ge n_2$ and $k\in \mathbb Z$. Put $w_1=g_nl_n^N$ and $w_2=g_nl_n^{N'+nk}$, note that $w_1\neq w_2$ as $l_n$ is of infinite order, $N\neq N'$ and $n>N+N'$. We can construct the $w_1$-, $w_2$-,  and $w_1w_2$-invariant $(B,\alpha)$-piecewise geodesics $\gamma_{w_1}$, $\gamma_{w_2}$ and $\gamma_{w_1w_2}$ with the same properties as before. Note that by Theorem~\ref{michael} (a) and (b) and as $C=B\ge 100$ we can further assume that the $1$-neighborhoods of $[x_n,y_n]$, $N_n$ and $g_nN_n$ do not intersect any translate of $N_n$ except $N_n$ and $g_nN_n$. Note that the existence of the piecewise geodesic $\gamma_{w_1w_2}$ on which $w_1w_2$ acts non-trivially implies that $w_1w_2\neq 1$, i.e. that $w_1\neq w_2^{- 1}$.

\smallskip Now the axes $A_{w_1}$ and $A_{w_2}$ are $\frac{1}{2}$-Hausdorff-close to $\gamma_{w_1}$ and $\gamma_{w_2}$. As the $1$-neighborhoods of $\gamma_{w_i}$ does not meet any translates of $N_n$ except the $w_i^kN_n$ this implies that the translates of $N_n$ intersected by the $\frac{1}{2}$-neighborhood $\bar A_{w_i}$ of $A_{w_i}$ are precisely the translates $w_i^kN_n$ for $i=1,2$. Note that $N_n$ and $g_nN_n$ are intersected by both  $\bar A_{w_1}$ and $\bar A_{w_2}$. Now if $w_1$ and $w_2^\varepsilon$ are conjugate then there must exist some $h\in\pi_1(M_n)$ such that $w_2^\varepsilon=hw_1h^{-1}$ which implies that   $hA_{w_1}=A_{w_2^\varepsilon}=A_{w_2}$; $h$ must in particular map the translates intersected by  $\bar A_{w_1}$ to those intersected by $\bar A_{w_2}$.
 
 After replacing $h$ by $hw_1^l$ for some $l\in\mathbb Z$ we can assume that $h$ fixes both $N_n$ and $gN_n$, indeed $h$ cannot exchange $N_n$ and $gN_n$ as it would otherwise fix the midpoint of $[x_n,y_n]$ and therefore be elliptic. Note that this replacement does not alter the fact that $w_2^\varepsilon=hw_1h^{-1}$.  As the intersection of the stabilizers of $N_n$ and $g_nN_n$ is trivial this implies that $h=1$ i.e. that $w_1=w_2^{\varepsilon}$. This is clearly a contradiction, thus $w_1$ and $w_2^{\varepsilon}$ are not conjugate. \end{proof}

\begin{proof}[Proof of Theorem~\ref{main}] Let $N_1$ and $N_2$ be as Lemma~\ref{lem1} and Lem\-ma~\ref{lem2}. Take $N_0:=\max(N_1,N_2)$ and  $N\neq N'\ge N_0$. Lemmas \ref{lem1} and \ref{lem2} then give numbers $n_1$ and $n_2$. Let $n_0:=\max(n_1,n_2)$. We show that $P_n^N$ and $P_n^{N'}$ are not Nielsen equivalent in $\pi_1(M_n)$ for any $n\ge n_0$. 

\smallskip
Since $P_n^{N'}$ is a generating pair for $G_n$, we have an epimorphism $\psi:F(a,b)\to G_n$ that takes $a$ to $m_n$ and $b$ to $g_nl_n^{N'}$. Suppose, for a contradiction, that $P_n^N$ and $P_n^{N'}$ are Nielsen equivalent. Hence there is an automorphism $\alpha:F(a,b)\to F(a,b)$ so that $\psi\circ\alpha(a)=m_n$ and $\psi\circ\alpha(b)=g_nl_n^N$. Let $b_1:=\alpha(a)$ and $b_2:=\alpha(b)$.

\smallskip
Using Proposition~\ref{osbornezieschang}, we have that $b_1$ is conjugate in $F(a,b)$ to a positive word $w$ in $a^{\varepsilon}$ and $b^\eta$ for some $\varepsilon,\eta\in\{-1,1\}$. Therefore we have $u\in F(a,b)$ such that $\psi(u)\psi(w)\psi(u^{-1})=m_n$. It follows from Lemma~\ref{lem1} that $w=a^\varepsilon$. As in the last part of the proof of Theorem~\ref{graphsofgroups}, we must have $b_2=ua^{k_1}b^{\sigma}a^{k_2}u^{-1}$ for some $k_1,k_2\in\mathbb Z$ and $\sigma\in\{-1,1\}$. Hence, $\psi(b_2)=g_nl_n^N$ is conjugate to $\left(g_nl_n^{N'-kn}\right)^{\nu}$ for some $k\in \mathbb Z$ and $\nu\in\{-1,1\}$. This contradicts Lemma~\ref{lem2}.
\end{proof}

\end{document}